\documentclass[11pt]{article}
\usepackage{graphicx}
\usepackage[a4paper,margin=27mm]{geometry}
\usepackage[T1]{fontenc}
\usepackage[utf8]{inputenc}
\usepackage{amsmath,amssymb,amsthm,mathtools,bm}
\usepackage{lmodern}
\usepackage{microtype}
\usepackage{booktabs,tabularx,array}
\usepackage{enumitem}
\usepackage{xcolor}
\usepackage{titlesec}
\usepackage{fancyhdr}
\usepackage{hyperref}
\usepackage[nameinlink,capitalise,noabbrev]{cleveref}

\definecolor{navy}{RGB}{22,57,95}
\definecolor{teal}{RGB}{0,108,117}
\newcommand{\REV}[1]{#1}
\hypersetup{colorlinks=true,linkcolor=navy,citecolor=teal,urlcolor=teal,
  pdftitle={Geometric Constants in Banach Tensor Products}}
\titleformat{\section}{\large\bfseries\color{navy}}{\thesection.}{0.65em}{}
\setlist{leftmargin=2em,itemsep=2pt,topsep=4pt}
\allowdisplaybreaks
\newtheorem{theorem}{Theorem}[section]
\newtheorem{proposition}[theorem]{Proposition}
\newtheorem{lemma}[theorem]{Lemma}
\newtheorem{corollary}[theorem]{Corollary}

\theoremstyle{definition}
\newtheorem{definition}[theorem]{Definition}
\newtheorem{example}[theorem]{Example}

\theoremstyle{remark}
\newtheorem{remark}[theorem]{Remark}

\crefname{theorem}{Theorem}{Theorems}
\crefname{proposition}{Proposition}{Propositions}
\crefname{lemma}{Lemma}{Lemmas}
\crefname{corollary}{Corollary}{Corollaries}
\crefname{definition}{Definition}{Definitions}
\crefname{example}{Example}{Examples}
\crefname{remark}{Remark}{Remarks}
\crefname{equation}{equation}{equations}
\crefname{figure}{Figure}{Figures}

\crefalias{example}{example}

\newcommand{\CNJ}{C_{\mathrm{NJ}}}
\newcommand{\tCNJ}{\widetilde C_{\mathrm{NJ}}}
\newcommand{\eps}{\varepsilon}
\newcommand{\oteps}{\widehat\otimes_{\eps}}
\newcommand{\otpi}{\widehat\otimes_{\pi}}
\newcommand{\Lop}{\mathcal L}

\newcommand{\R}{\mathbb R}

\usepackage{authblk}

\title{\bfseries Von Neumann–Jordan Constants in Banach Tensor Products}

\author[1]{Wenwen Zhang}
\author[1]{Ying Xu}
\author[1]{Qi Liu\thanks{Corresponding author: liuq67@aqnu.edu.cn}}
\author[2]{Yongjin Li}

\affil[1]{School of Mathematics and Statistics, Anqing Normal University,
	Anqing 246133, China}

\affil[2]{Department of Mathematics, Sun Yat-sen University,
	Guangzhou 510275, China}
\date{}

\begin{document}
	
	\maketitle
	
	\begin{center}
		Email: Y26220004@stu.aqnu.edu.cn; 
		Y26220046@stu.aqnu.edu.cn;
		liuq67@aqnu.edu.cn; 
		stslyj@mail.sysu.edu.cn
	\end{center}

\begin{abstract}
	We investigate the von Neumann--Jordan constant of Banach tensor products equipped with the injective and projective tensor norms. We first establish lower estimates in terms of the geometric properties of the factor spaces and show that, whenever both factors are infinite-dimensional, the von Neumann--Jordan constants of both canonical tensor products attain the maximal value \(2\). For finite-dimensional factors, we obtain explicit results for classical sequence spaces and characterize the extremal case through operator-space and orthogonal-contact conditions. We further derive dimension-sensitive estimates based on Euclidean sections. Finally, we study the constant under equivalent renormings and determine its exact renorming envelope for injective and projective tensor products. As consequences, we characterize when these tensor products are superreflexive and when they are weak Hilbert spaces. These results clarify the influence of tensor norms, dimension, and renorming on the von Neumann--Jordan geometry of Banach tensor products.
\end{abstract}



\section{Introduction}

The parallelogram identity provides a fundamental characterization of inner product spaces and serves as a basic tool for describing the geometry of Banach spaces.Motivated by this characterization, Clarkson introduced the von Neumann--Jordan constant to quantify the deviation of a Banach space from the parallelogram identity\cite{Clarkson1937}. For a Banach space $X$, the constant satisfies
$1\leq \CNJ(X)\leq2$, where the equality $\CNJ(X)=1$ characterizes Hilbert
spaces, while $\CNJ(X)=2$ corresponds to the maximal possible value of this constant.
\cite{Clarkson1937,JordanVonNeumann1935}.

Subsequent studies have investigated the relations between the von Neumann--Jordan constant and various geometric properties of Banach spaces, including uniform non-squareness, normal structure, and related geometric properties of Banach spaces. In particular,
estimates involving this constant have been used to characterize uniform
non-squareness, normal structure, and other geometric features of Banach
spaces
\cite{KatoMaligrandaTakahashi2001,DhompongsaPiraisangjunSaejung2003}.
Moreover, explicit computations and sharp estimates for special classes of
spaces, including finite-dimensional spaces and classical function spaces,
have further clarified the geometric meaning of this constant
\cite{HashimotoNakamura2009,Mizuguchi2021}.

It is worth emphasizing that the von Neumann--Jordan constant plays an
important role in the study of the geometry of Banach spaces and has led to
a variety of interesting applications and results. In particular, it has been
investigated in connection with the structure of Banach spaces under equivalent
renormings, approximate versions of the Jordan--von Neumann theorem for
finite-dimensional real normed spaces, and its relationship with other
geometric constants, such as the James constant
\cite{HashimotoNakamura2009,Passer2015,Wang2010,KatoTakahashi2009}.
Furthermore, the von Neumann--Jordan constant has been extensively studied for
various classes of Banach spaces, where it provides valuable information on
their geometric structures. Explicit estimates and exact values have been
obtained for several important spaces, including Day--James spaces,
Bana\`s--Fr\c{a}czek spaces, and Lorentz sequence spaces
\cite{YangWang2016,Yang2014,KatoMaligrandaTakahashi2001}.
Moreover, its connections with other geometric properties, such as roundness
and uniform normal structure, have been investigated through generalized forms
of the Jordan--von Neumann constant
\cite{AminiHarandiDoustRobertson2021,DhompongsaPiraisangjunSaejung2003}.
These results demonstrate the effectiveness of the von Neumann--Jordan constant
as a tool for analyzing the nonlinear geometry of Banach spaces. For specific
results and further developments, we refer the reader to the relevant references.

Tensor products provide a fundamental framework for constructing new Banach spaces, and the theory of tensor norms has been systematically developed in the classical monographs \cite{Ryan2002,DefantFloret1993}.  Different tensor norms may produce essentially different norm geometries on the resulting tensor product spaces. In particular, the injective and projective
tensor norms represent two important tensor constructions, and the behavior
of the von Neumann--Jordan constant under these norms is not immediate.
Although the von Neumann--Jordan constant has been extensively studied in individual Banach spaces, much less is known about its behavior in tensor product spaces. In particular, it remains unclear how the choice of tensor norms and the geometric properties of the factor spaces influence this constant.These questions motivate our investigation of how the geometry of factor spaces is reflected in their tensor products.

The paper is organized as follows. Section 2 collects the notation and
preliminary facts about the von Neumann--Jordan constant, tensor norms and
related geometric quantities. The subsequent sections develop the main
results on finite-dimensional rigidity, infinite-dimensional extremality,
and further geometric consequences of tensor products.

\section{Preliminaries}

\REV{Throughout the paper, $X$ and $Y$ denote Banach spaces over 
	$\mathbb K$, where $\mathbb K=\mathbb R$ or $\mathbb C$. Tensor products are
	understood with respect to their canonical completions. We begin with the
	von Neumann--Jordan constant}

\begin{equation}\label{eq:def-cnj}
	\CNJ(X)=\sup
	\left\{
	\frac{\|x+y\|^{2}+\|x-y\|^{2}}
	{2(\|x\|^{2}+\|y\|^{2})}
	:\ x,y\in X,\ (x,y)\ne(0,0)
	\right\}.
\end{equation}

\REV{The quantity $\CNJ(X)$ is referred to as the von Neumann--Jordan constant.
	This constant and its basic properties have been extensively studied in
	Banach space theory; see, for example,
	\cite{Clarkson1937,JordanVonNeumann1935,KatoTakahashi1997}.
	We recall several fundamental properties that will be used later:
	\begin{enumerate}[label=\textup{(\roman*)}]
		\item $1\leq \CNJ(X)\leq 2$.
		\item $X$ is a Hilbert space if and only if $\CNJ(X)=1$.
		\item $X$ is uniformly non-square if and only if $\CNJ(X)<2$.
		\item $\CNJ(X)=\CNJ(X^*)$.
\end{enumerate}}

We shall also use the James constant and the Sch\"affer constant, defined by
\[
J(X)=\sup_{x,y\in S_X}\min\{\|x+y\|,\|x-y\|\},
\qquad
S(X)=\inf_{x,y\in S_X}\max\{\|x+y\|,\|x-y\|\},
\]
which are related by
\[
J(X)S(X)=2.
\]
\REV{We shall use the relation above together with the standard estimates
	recorded in \cite{James1964,KatoMaligrandaTakahashi2001}.}

\REV{We first record several standard properties of the von Neumann--Jordan
	constant.}

\begin{lemma}\label{lem:basic-cnj}
	If $E$ is an isometric subspace of $X$, then
	\[
	\CNJ(E)\le\CNJ(X).
	\]
	Moreover,
	\[
	\CNJ(X)=\CNJ(X^\ast),
	\qquad
	\CNJ(X)=\sup_{\substack{E\subset X\\ \dim E\le2}}\CNJ(E).
	\]
\end{lemma}

\begin{lemma}\label{lem:endpoint-normalization}
	For every $X$, $\CNJ(X)=2$ if and only if there exist
	$x_k,y_k\in S_X$ such that
	\[
	\|x_k+y_k\|\longrightarrow2,
	\qquad
	\|x_k-y_k\|\longrightarrow2 .
	\]
\end{lemma}

\begin{proof}
	\REV{By \eqref{eq:def-cnj}, the quotient for unit vectors is at most $2$.
		If $\CNJ(X)=2$, choose a sequence of pairs whose quotients tend to $2$;
		then both $\|x_k+y_k\|$ and $\|x_k-y_k\|$ tend to $2$. The converse is
		immediate from the same formula.}
\end{proof}

We now define the injective and projective tensor norms on the algebraic
tensor product $X\otimes Y$. For
$w=\sum_{j=1}^r x_j\otimes y_j\in X\otimes Y$, define
\begin{align}
	\|w\|_\varepsilon
	&=\sup\left\{
	\left|\sum_{j=1}^r x^\ast(x_j)y^\ast(y_j)\right|:
	x^\ast\in B_{X^\ast},y^\ast\in B_{Y^\ast}\right\},\\
	\|w\|_\pi
	&=\inf\left\{
	\sum_{j=1}^r\|x_j\|\|y_j\|:
	w=\sum_{j=1}^r x_j\otimes y_j\right\}.
\end{align}
\REV{For $\alpha\in\{\varepsilon,\pi\}$, both tensor norms are crossnorms; in
	particular,}
\[
\|x\otimes y\|_\alpha=\|x\|\|y\|.
\]
We shall also use the canonical isometric identification
\[
(X\widehat\otimes_\pi Y)^*
\cong\mathcal L(X,Y^\ast).
\]

For more details on tensor products, see~\cite{Ryan2002}.

\section{Finite-dimensional tensor structures and von Neumann--Jordan constants}

\REV{We first examine finite-dimensional configurations that force extremal
behavior of the tensor constants.}

The following result shows that, although both factors possess exact Hilbertian geometry, the two canonical Banach-space tensor norms destroy this structure to the maximal extent measured by the von Neumann--Jordan constant.

\begin{proposition}\label{thm:hilbert-cert}
If $H,K$ are Hilbert spaces of dimension at least two, then
\[
 \CNJ(H\oteps K)=\CNJ(H\otpi K)=2.
\]
\end{proposition}

\begin{proof}
Choose orthonormal $e_1,e_2\in H$ and $f_1,f_2\in K$. For the injective norm set
\[
 U=e_1\otimes f_1+e_2\otimes f_2,
 \qquad V=e_1\otimes f_1-e_2\otimes f_2.
\]
Under the standard operator-space identification of the injective tensor
product, $\|U\|_\eps=\|V\|_\eps=1$ and
$\|U\pm V\|_\eps=2$. For the projective norm set
$A=e_1\otimes f_1$ and $B=e_2\otimes f_2$. The crossnorm property gives $\|A\pm B\|_\pi\le2$. Equality follows by
evaluating against norm-one operators that map $e_1$ to $f_1$ and
$e_2$ to $\pm f_2$.
Thus both defining quotients are equal to $2$.
\end{proof}

\begin{corollary}\label{cor:non-hilbert-tensor-obstruction}
Let $X$ and $Y$ be non-zero Banach spaces satisfying
$\dim X\ge2$ and $\dim Y\ge2$.  Then, for
$\alpha\in\{\eps,\pi\}$,
\[
 \CNJ(X\widehat\otimes_\alpha Y)>1.
\]
Consequently, neither $X\widehat\otimes_\eps Y$ nor
$X\widehat\otimes_\pi Y$ is isometric to a Hilbert space.
\end{corollary}

\begin{proof}
If the tensor product were Hilbert, its norm would satisfy the parallelogram
identity and its von Neumann--Jordan constant would be $1$. Choose two linearly independent vectors in each factor. The resulting
two-dimensional tensor configurations violate the parallelogram identity for
each of the two canonical tensor norms. Therefore, the tensor norm is not induced by an inner product. The
Jordan--von Neumann characterization then yields
\[
 \CNJ(X\widehat\otimes_\alpha Y)\ne1 .
\]
Since every normed space satisfies $\CNJ\ge1$, the strict inequality follows.
\end{proof}

\REV{For $d\ge2$, we quantify the Euclidean structure of $d$-dimensional
	subspaces of $X$ by
\[
 \eta_d(X)=\inf\{d(E,\ell_2^d): E\subset X,\ \dim E=d\},
\]
where $\eta_d(X)=\infty$ if $X$ contains no $d$-dimensional subspace. We also use the complemented Euclidean-section index
\[
 \kappa_r(X)=\inf\{\|P\|\,d(E,\ell_2^r): E\subset X,\ \dim E=r,\ P:X\to E\text{ is a projection}\},
\]
with $\kappa_r(X)=\infty$ if the set over which the infimum is taken is empty.}

\begin{theorem}\label{thm:two-plane}
If both factors have dimension at least two, then
\begin{align}
 \CNJ(X\oteps Y)&\ge
 \max\left\{\CNJ(X),\CNJ(Y),
 \frac{2}{\eta_2(X)^2\eta_2(Y)^2}\right\},\label{eq:epsbound}\\
 \CNJ(X\otpi Y)&\ge
 \max\left\{\CNJ(X),\CNJ(Y),
 \frac{2}{\eta_2(X^*)^2\eta_2(Y^*)^2}\right\}.\label{eq:pibound}
\end{align}
\end{theorem}

\begin{proof}
Choose two-dimensional subspaces $E\subset X$ and $F\subset Y$ that are
arbitrarily close to attaining the infima defining $\eta_2(X)$ and
$\eta_2(Y)$. The injective tensor norm, the corresponding
Banach--Mazur distortion estimate, and \cref{thm:hilbert-cert} yield
\eqref{eq:epsbound}. For the projective estimate, apply the injective
estimate to $X^*$ and $Y^*$, using the dual identification of projective and
injective tensor products and the identity $\CNJ(Z)=\CNJ(Z^*)$.
\end{proof}

\begin{theorem}\label{thm:rigidity}
If $X$ and $Y$ are infinite-dimensional, then
\[
\CNJ(X\oteps Y)=\CNJ(X\otpi Y)=2.
\]
Consequently, if the two tensor constants are different, at least one factor is finite-dimensional.
\end{theorem}

\begin{proof}
Dvoretzky's theorem gives $\eta_2(X)=\eta_2(Y)=\eta_2(X^*)=\eta_2(Y^*)=1$.  Apply \cref{thm:two-plane} and the universal upper bound.  The final assertion is the contrapositive.
\end{proof}

The following remark provides a new argument.

\begin{remark}

Suppose first that $C_{\mathrm{NJ}}(X)=2$ or $C_{\mathrm{NJ}}(Y)=2$. \cref{thm:two-plane} and the universal upper bound $C_{\mathrm{NJ}}\leq 2$ immediately give the conclusion for both tensor norms.

Assume that $X$ and $Y$ are infinite-dimensional and satisfy
$\CNJ(X)<2$ and $\CNJ(Y)<2$. Then both spaces are uniformly non-square and,
consequently, superreflexive.Suppose, for contradiction, that
\[
\CNJ(X\widehat\otimes_\pi Y)<2.
\]
Then $X\widehat\otimes_\pi Y$ would be uniformly non-square and hence
superreflexive. By Rueda Zoca's theorem \cite{RuedaZoca2025}, for
	superreflexive factors the
	projective tensor product is superreflexive if and only if at least one
	factor is finite-dimensional, which contradicts the infinite-dimensional
	assumption.  The injective case follows analogously.  Thus, for infinite-dimensional factors, both canonical tensor norms yield
	the maximal value $\CNJ=2$.
\end{remark}

\begin{corollary}
	If
	\[
	\CNJ(X\widehat\otimes_\eps Y)\neq
	\CNJ(X\widehat\otimes_\pi Y),
	\]
	then at least one of $X,Y$ is finite-dimensional.
\end{corollary}

\begin{theorem}\label{thm:product-upper-bound-infinite}
Let $\mathbb K\in\{\R,\mathbb C\}$, let $X$ and $Y$ be infinite-dimensional Banach spaces over $\mathbb K$, and let $\alpha\in\{\eps,\pi\}$.  Then the following are equivalent:
\begin{enumerate}[label=\textup{(\roman*)}]
\item
\[
 \CNJ(X\widehat\otimes_\alpha Y)
 \le \CNJ(X)\CNJ(Y);
\]
\item
\[
 \CNJ(X)\CNJ(Y)\ge2.
\]
\end{enumerate}
Equivalently,
\[
 \max\{\CNJ(X),\CNJ(Y)\}
 \le \CNJ(X\widehat\otimes_\alpha Y)=2
 \le \CNJ(X)\CNJ(Y)
\]
holds exactly when $\CNJ(X)\CNJ(Y)\ge2$. This is an immediate consequence
of the extremal identity in \cref{thm:rigidity}.
\end{theorem}

\begin{proof}
By \cref{thm:rigidity},
\[
 \CNJ(X\widehat\otimes_\alpha Y)=2.
\]
Hence the product upper bound in \textup{(i)} is equivalent to
$2\le \CNJ(X)\CNJ(Y)$, which is precisely \textup{(ii)}.  The left-hand inequality follows from the canonical isometric embeddings of
the two factors.
\end{proof}

The different restrictions on the parameters in the following theorem also reflect the intrinsic distinction between the two canonical tensor norms.

\begin{theorem}\label{thm:phase}
Let $1\le p,q\le\infty$ and $m,n\ge2$.
\begin{enumerate}[label=\textup{(\roman*)}]
\item If $p^{-1}+q^{-1}\le1$, then
\[
 \CNJ(\ell_p^m\oteps\ell_q^n)=2.
\]
\item For every $p,q$,
\[
 \CNJ(\ell_p^m\otpi\ell_q^n)=2.
\]
\end{enumerate}
Thus the projective tensor product always attains the maximal value $2$,
whereas the injective conclusion in \textup{(i)} holds in the region
$p^{-1}+q^{-1}\le1$.
\end{theorem}

\begin{proof}
If $p$ or $q$ is an endpoint exponent, $1$ or $\infty$, then the
corresponding factor has von Neumann--Jordan constant $2$, and the canonical
factor embedding gives the required conclusion. We may therefore assume
$1<p,q<\infty$.

For (i), put
\[
 U=e_1\otimes f_1+e_2\otimes f_2,
 \qquad V=e_1\otimes f_1-e_2\otimes f_2.
\]
The hypothesis is $p'\le q$.  If $a\in B_{\ell_{p'}^m}$ and $b\in B_{\ell_{q'}^n}$, H\"older's inequality and monotonicity of finite-dimensional $\ell_r$ norms give
\[
 |a_1b_1\pm a_2b_2|
 \le\|(a_1,a_2)\|_q\|(b_1,b_2)\|_{q'}
 \le1.
\]
Thus $\|U\|_\eps=\|V\|_\eps=1$ and $\|U\pm V\|_\eps=2$.

For (ii), first suppose $p^{-1}+q^{-1}\ge1$.  Take $A=e_1\otimes f_1$ and $B=e_2\otimes f_2$.  Since $p\le q'$, the bilinear forms
\[
 b_\pm(x,y)=x_1y_1\pm x_2y_2
\]
have norm one on $\ell_p^m\times\ell_q^n$.  They show $\|A\pm B\|_\pi=2$, while $\|A\|_\pi=\|B\|_\pi=1$.

It remains to treat the genuinely subcritical case
\[
 s:=p^{-1}+q^{-1}<1.
\]
Consider the following off-diagonal construction: Set
\[
 a=e_1+e_2,\qquad a'=e_1-e_2,\qquad
 b=f_1+f_2,\qquad b'=f_1-f_2,
\]
and
\[
 x=a\otimes b,\qquad y=a'\otimes b'.
\]
The crossnorm identity gives
\[
 \|x\|_\pi=\|y\|_\pi=2^s.
\]
Write
\[
 d=e_1\otimes f_1+e_2\otimes f_2,
 \qquad
 d'=e_1\otimes f_2+e_2\otimes f_1.
\]
Since
\[
 d=\frac12\bigl((e_1+e_2)\otimes(f_1+f_2)
 +(e_1-e_2)\otimes(f_1-f_2)\bigr),
\]
one has $\|d\|_\pi\le2^s$.  Choose $t\in[1,\infty]$ so that
\[
 p^{-1}+q^{-1}+t^{-1}=1
\]
and define
\[
 \beta(u,v)=2^{s-1}(u_1v_1+u_2v_2).
\]
Three-factor H\"older gives $\|\beta\|\le1$, hence
\[
 \|d\|_\pi\ge \beta(d)=2^s.
\]
Therefore $\|d\|_\pi=2^s$.  Coordinate transposition in the second factor is an isometry, so also $\|d'\|_\pi=2^s$.  Since
\[
 x+y=2d,\qquad x-y=2d',
\]
we obtain
\[
 \|x+y\|_\pi=\|x-y\|_\pi=2^{s+1}.
\]
Consequently, the defining von Neumann--Jordan quotient of the pair
$(x,y)$ is equal to $2$.
\end{proof}

\begin{remark}\label{rem:off-diagonal-projective}
In the subcritical region $p^{-1}+q^{-1}<1$, the projective endpoint above is
not forced by the usual $\ell_1$ diagonal. The extremizing rank-one pair is
\[
 (e_1+e_2)\otimes(f_1+f_2),
 \qquad
 (e_1-e_2)\otimes(f_1-f_2),
\]
whose sum and difference produce diagonal and anti-diagonal $2\times2$
Hadamard patterns. This is precisely the off-diagonal Hadamard construction used in the proof.
\end{remark}

Let $E$ be finite-dimensional and $Z$ a Banach space.  Write $\operatorname{Sq}(E,Z)$ when there exist $A_k,B_k\in S_{\Lop(E,Z)}$ such that
\[
 \|A_k+B_k\|\to2,
 \qquad \|A_k-B_k\|\to2.
\]
If $Y$ is finite-dimensional, the canonical identifications give
\begin{equation}\label{eq:mixed-operator-models}
 X\oteps Y\cong\Lop(Y^*,X),
 \qquad
 \CNJ(X\otpi Y)=\CNJ(\Lop(Y,X^*)).
\end{equation}

\begin{theorem}\label{thm:fixed-pair}
Let $Y$ be finite-dimensional.  Then
\begin{align*}
 \CNJ(X\oteps Y)=2&\quad\Longleftrightarrow\quad
 \operatorname{Sq}(Y^*,X),\\
 \CNJ(X\otpi Y)=2&\quad\Longleftrightarrow\quad
 \operatorname{Sq}(Y,X^*).
\end{align*}
Equivalently, each condition is witnessed by an exact operator square after passing to an ultrapower of the range.
\end{theorem}

\begin{proof}
Combine \eqref{eq:mixed-operator-models} with \cref{lem:endpoint-normalization}.  For an ultrapower $Z_{\mathcal U}$, finite dimensionality of $E$ gives the canonical isometry
\[
 (\Lop(E,Z))_{\mathcal U}\cong\Lop(E,Z_{\mathcal U}).
\]
Taking classes of an approximating square gives exact norm-two sums and differences; representatives give the converse.
\end{proof}

\REV{We next describe the same phenomenon through orthogonal contacts with
Euclidean spaces.}

If $d=\dim E\ge2$, define
\begin{equation}\label{eq:omega-perp}
 \omega_\perp(E)=\max\left\{
 \min\{\|Qu\|_2,\|Qv\|_2\}:
 \begin{array}{l}
 Q:E\to\ell_2^d,\ \|Q\|\le1,\ u,v\in S_E,\\[-1mm]
 \langle Qu,Qv\rangle=0
 \end{array}\right\}.
\end{equation}
We say that $E$ has the \emph{orthogonal-contact property} (OCP) if $\omega_\perp(E)=1$.

\begin{theorem}\label{thm:ocp}
Let $E$ be finite-dimensional with $\dim E\ge2$ and let $Z$ be infinite-dimensional.  Then
\begin{equation}\label{eq:ocp-bound}
 \CNJ(\Lop(E,Z))\ge
 \max\{\CNJ(Z),\CNJ(E^*),2\omega_\perp(E)^2\}.
\end{equation}
In addition, the following three conditions are equivalent:
\begin{enumerate}[label=\textup{(\roman*)}]
\item $E$ has OCP;
\item $\CNJ(\Lop(E,H))=2$ for an infinite-dimensional Hilbert space $H$;
\item $\CNJ(\Lop(E,Z))=2$ for every infinite-dimensional $Z$.
\end{enumerate}
Consequently, for $2\le\dim Y<\infty$,
\begin{align*}
 \CNJ(X\oteps Y)=2\ \text{for every infinite-dimensional }X
 &\Longleftrightarrow \omega_\perp(Y^*)=1,\\
 \CNJ(X\otpi Y)=2\ \text{for every infinite-dimensional }X
 &\Longleftrightarrow \omega_\perp(Y)=1.
\end{align*}
In particular, for every infinite-dimensional Hilbert space $H$,
\begin{align*}
 \CNJ(H\oteps Y)=2&\Longleftrightarrow \omega_\perp(Y^*)=1,\\
 \CNJ(H\otpi Y)=2&\Longleftrightarrow \omega_\perp(Y)=1.
\end{align*}
\end{theorem}

\begin{proof}
Fix $t<\omega_\perp(E)$ and choose a contraction $Q:E\to H_0$ and $u,v\in S_E$ such that $Qu\perp Qv$ and both image norms exceed $t$.  Let $R$ be the orthogonal reflection fixing $Qu$ and reversing $Qv$.  By Dvoretzky's theorem, for every $\delta>0$ there exists an operator
$J:H_0\to Z$ satisfying
\[
 (1+\delta)^{-1}\|h\|_2\le\|Jh\|\le\|h\|_2.
\]
For $A=JQ$ and $B=JRQ$, one has $\|A\|,\|B\|\le1$ and
\[
 \|A+B\|\ge2(1+\delta)^{-1}t,
 \qquad
 \|A-B\|\ge2(1+\delta)^{-1}t.
\]
Insertion in \eqref{eq:def-cnj}, followed by $\delta\downarrow0$ and $t\uparrow\omega_\perp(E)$, proves the orthogonal-contact term in \eqref{eq:ocp-bound}; the factor terms come from rank-one operator embeddings.

OCP therefore implies (iii), and (iii) trivially implies (ii).  Conversely, assume
\[
 \CNJ(\Lop(E,H))=2.
\]
Choose $S_n,T_n\in S_{\Lop(E,H)}$ with $\|S_n\pm T_n\|\to2$, and set
\[
 A_n=\frac{S_n+T_n}{2},\qquad B_n=\frac{S_n-T_n}{2}.
\]
Then $\|A_n\|\to1$, $\|B_n\|\to1$, while $\|A_n+B_n\|=\|S_n\|=1$ and $\|A_n-B_n\|=\|T_n\|=1$.  Fix a basis $e_1,\ldots,e_d$ of $E$.  The Gram matrices of
\[
 A_ne_1,\ldots,A_ne_d,B_ne_1,\ldots,B_ne_d
\]
form a bounded family of positive semidefinite matrices.  Passing to a subsequence, they converge entrywise to a positive semidefinite matrix $G$.  Realize $G$ as the Gram matrix of vectors in a finite-dimensional Hilbert space $K$, and use those vectors to define operators $A,B:E\to K$. Since $E$ is finite-dimensional, convergence of the associated quadratic
forms is uniform on $S_E$. Consequently,
\[
 \|A\|=\|B\|=1,\qquad \|A+B\|\le1,\qquad \|A-B\|\le1.
\]
Choose $u,v\in S_E$ with $\|Au\|=1$ and $\|Bv\|=1$.  The parallelogram identity yields, for every $e\in E$,
\[
 \|Ae\|^2+\|Be\|^2
 =\frac12\bigl(\|(A+B)e\|^2+\|(A-B)e\|^2\bigr)
 \le \|e\|^2.
\]
Therefore $Bu=0$ and $Av=0$.  The map
\[
 Q_0e=(Ae,Be)\in K\oplus_2 K
\]
is a contraction and sends $u,v$ to orthogonal unit vectors.  Orthogonally projecting onto $\operatorname{span}\{Q_0u,Q_0v\}$ and identifying this plane with $\ell_2^2\subset\ell_2^d$ gives the contraction required in \eqref{eq:omega-perp}.  Hence $E$ has OCP.  The tensor statements follow from \eqref{eq:mixed-operator-models}.
\end{proof}

\begin{example}\label{ex:ocp-examples}
let $H_6=\R^2$ with
\[
 \|(a,b)\|_{H_6}=\max\{|a|,|b|,|a+b|\}.
\]
Then $H_6^*$ is isometric to $H_6$, while
\[
 \CNJ(H_6)<2,
 \qquad \omega_\perp(H_6)<1.
\]
Consequently, if $H$ is infinite-dimensional Hilbert,
\[
 \CNJ(H\oteps H_6)<2,
 \qquad \CNJ(H\otpi H_6)<2.
\]
\end{example}

\begin{proof}
For $H_6$, suppose a Hilbert contraction had two orthogonal unit contact points.  Strict convexity of the Hilbert ball allows both contact points to be replaced by vertices of the hexagon.  Up to signs, two independent vertices come from $e_1,e_2,e_1-e_2$; the remaining vertex is their sum or difference, whose image would have Hilbert norm $\sqrt2$, contradicting contractivity.  Compactness gives $\omega_\perp(H_6)<1$.

It remains to verify $\CNJ(H_6)<2$.  Otherwise finite-dimensional compactness and \cref{lem:endpoint-normalization} would give $x,y\in S_{H_6}$ with
\[
 \|x+y\|=\|x-y\|=2.
\]
Choose norming functionals for $x+y$ and $x-y$.  They show, respectively in the same-sign and opposite-sign quadrants, that
\[
 \|\alpha x+\beta y\|=|\alpha|+|\beta|
 \qquad(\alpha,\beta\in\R).
\]
Hence $H_6$ would be linearly isometric to $\ell_1^2$, whose unit ball is a parallelogram, contradicting the six extreme points of the hexagonal unit ball.  Therefore $\CNJ(H_6)<2$.  Apply \cref{thm:ocp} and self-duality.
\end{proof}

Finally, we demonstrate that the two tensor norms may exhibit genuinely different geometric behavior. The following example highlights this separation in a minimal and explicit setting.

The two tensor constants need not even be ordered.  The following example is dimensionally minimal and uses only rational data.

\begin{example}\label{ex:separation}
Let $X=\R^2$ with
\[
 \|(s,t)\|_X=\max\{|s|,|t|,|s-t|\},
\]
and let $Y=\R^2$ with
\[
 \|(s,t)\|_Y=\max\left\{|s|,
 \left|\frac45s+\frac35t\right|,
 \left|\frac35s+\frac45t\right|,|t|,
 \left|-\frac{12}{25}s+\frac{16}{25}t\right|\right\}.
\]
Then
\[
 \CNJ(X\oteps Y)=2,
 \qquad
 \CNJ(X\otpi Y)<2.
\]
For the dual pair the strict inequality is reversed.  Moreover,
\[
 \CNJ(X)<2,
 \qquad
 \CNJ(Y)<2,
\]
so both factors are uniformly non-square.
\end{example}

\begin{proof}
Write
\[
 a_1=(1,0),\ a_2=(0,1),\ a_3=(1,-1)
\]
and
\[
 b_1=(1,0),\ b_2=\left(\frac45,\frac35\right),\
 b_3=\left(\frac35,\frac45\right),\ b_4=(0,1),\
 b_5=\left(-\frac{12}{25},\frac{16}{25}\right).
\]
These are, up to signs, the vertices of $B_{X^*}$ and $B_{Y^*}$.  Under the matrix identification,
\[
 \|M\|_\eps=\max_{i,j}|a_i^TMb_j|.
\]
Set
\[
 U=\begin{pmatrix}-1&-1/3\\-2/3&2/3\end{pmatrix},
 \qquad
 V=\begin{pmatrix}-1&13/16\\-1/4&-3/16\end{pmatrix}.
\]
Exact substitution gives $\|U\|_\eps=\|V\|_\eps=1$, while
\[
 a_1^T(U+V)b_1=-2,
 \qquad
 a_3^T(U-V)b_4=-2.
\]
Thus $\|U+V\|_\eps=\|U-V\|_\eps=2$ and the injective constant is two.

For projective strictness use finite-dimensional duality:
\[
 \CNJ(X\otpi Y)=\CNJ(X^*\oteps Y^*).
\]
The vertices of $B_X$ are, up to signs,
\[
 x_1=(1,0),\quad x_2=(1,1),\quad x_3=(0,1),
\]
and those of $B_Y$ are
\[
 y_1=\left(1,\frac13\right),\quad
 y_2=\left(\frac57,\frac57\right),\quad
 y_3=\left(\frac13,1\right),\quad
 y_4=\left(-\frac34,1\right),\quad
 y_5=\left(-1,\frac{13}{16}\right).
\]
Hence $Z=X^*\oteps Y^*$ has the polyhedral norm
\[
 \|M\|_Z=\max_{1\le i\le3,\,1\le j\le5}|x_i^TMy_j|.
\]
For a finite-dimensional polyhedral space with norming functionals $f_k$, the equality $\CNJ=2$ holds exactly when there exist $k,\ell$ and unit vectors $x,y$ with
\[
 f_k(x)=f_k(y)=1,
 \qquad f_\ell(x)=1,\quad f_\ell(y)=-1.
\]
Indeed, these equalities are precisely the simultaneous facet contacts that force $\|x+y\|=\|x-y\|=2$.

Enumerating the vertices determined by the thirty rational signed inequalities above gives the following exact incidence certificate.  For a fixed first facet $F_{1j}^+$, let $\mathcal P_j$ and $\mathcal N_j$ be the second facets meeting it with equal and opposite signs, respectively:
\begin{center}
\small
\begin{tabular}{@{}c l l@{}}
\toprule
$j$&$\mathcal P_j$&$\mathcal N_j$\\
\midrule
1&$12,21,22,23,34,35$&$15,24,25,31,32,33$\\
2&$11,13,21,22,23,24,35$&$25,31,32,33,34$\\
3&$12,14,21,22,23,24,25$&$31,32,33,34,35$\\
4&$13,15,22,23,24,25,31$&$21,32,33,34,35$\\
5&$14,23,24,25,31,32$&$11,21,22,33,34,35$\\
\bottomrule
\end{tabular}
\end{center}
Every row has $\mathcal P_j\cap\mathcal N_j=\varnothing$.  The isometry
\[
 \begin{pmatrix}1&-1\\1&0\end{pmatrix}
\]
of $X$ cycles the first facet index up to sign, so the table covers all cases.  The facet criterion therefore yields $\CNJ(Z)<2$.  Finally,
\[
 \CNJ(X^*\oteps Y^*)=\CNJ(X\otpi Y)<2,
 \qquad
 \CNJ(X^*\otpi Y^*)=\CNJ(X\oteps Y)=2,
\]
which proves the dual reversal. Finally, the canonical factor embeddings and
the strict projective inequality give
\[
 \max\{\CNJ(X),\CNJ(Y)\}\le \CNJ(X\otpi Y)<2.
\]
Hence both factors are uniformly non-square, using the standard equivalence
$\CNJ(E)<2$ $\Longleftrightarrow$ $E$ is uniformly non-square.
\end{proof}

The unit balls and the corresponding two-dimensional tensor sections are displayed in \cref{fig:rational-separation}.

\begin{figure}[!htbp]
\centering
\includegraphics[width=0.86\textwidth]{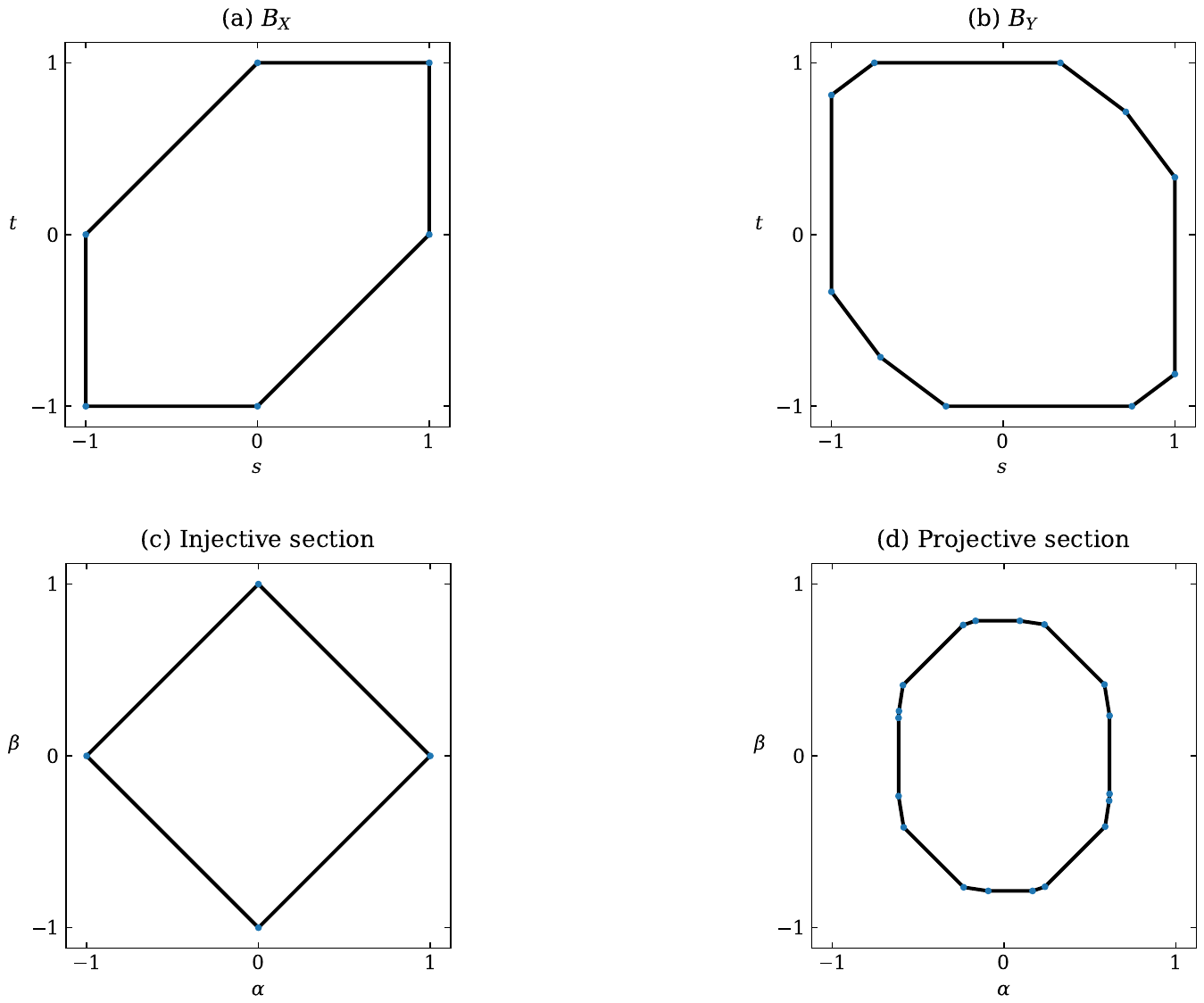}
\caption{The unit balls $B_X$ and $B_Y$, together with the two-dimensional tensor sections associated with the rational pair in \cref{ex:separation}. The injective section displays the endpoint certificate, while the projective section visualizes the strictness obstruction verified by facet incidence.}
\label{fig:rational-separation}
\end{figure}

\begin{remark}
The finite certificate is reproducible without numerical tolerances.  A vertex is obtained by solving four linearly independent active signed equations among fifteen absolute constraints; infeasible systems are discarded and the remaining rational vertices determine every nonempty face incidence.
\end{remark}

\section{Dimension-sensitive tensor estimates and geometric profiles}

\REV{We next make the dependence on dimension and Euclidean structure
explicit.}

Define the worst Euclidean two-plane index by
\[
 \Theta_2(Z)=\sup\{d(E,\ell_2^2):E\subset Z,\ \dim E=2\},
\]
with $\Theta_2(Z)=1$ when $\dim Z=1$.  John's theorem gives $1\le\Theta_2(Z)\le\sqrt2$ \cite{John1948}.

\begin{theorem}\label{thm:bilateral}
For non-zero $X,Y$,
\begin{align*}
 \max\left\{\CNJ(X),\CNJ(Y),
 \frac{2}{\eta_2(X^*)^2\eta_2(Y^*)^2}\right\}
 &\le \CNJ(X\otpi Y)\\
 &\le \Theta_2(\Lop(X,Y^*))^2\le2,
\end{align*}
where $1/\infty=0$.
\end{theorem}

\begin{proof}
The lower estimate is \eqref{eq:pibound}. For every two-dimensional
$E\subset Z$, the standard Banach--Mazur estimate gives
$\CNJ(E)\le d(E,\ell_2^2)^2$. Two-dimensional localization therefore yields
$\CNJ(Z)\le\Theta_2(Z)^2$. Apply this to
$Z=\Lop(X,Y^*)\cong(X\otpi Y)^*$ and use self-duality.
\end{proof}

\begin{corollary}\label{cor:factor-only-hilbertian-distortion}
Let $X$ and $Y$ admit onto isomorphisms $T:X\to H$ and $S:Y\to K$ onto Hilbert spaces with
$\dim H,\dim K\ge2$.  Put $c_X=\CNJ(X)$ and $c_Y=\CNJ(Y)$, and assume
\[
 \|T\|\,\|T^{-1}\|\le \sqrt{c_X},
 \qquad
 \|S\|\,\|S^{-1}\|\le \sqrt{c_Y}.
\]
Then
\begin{equation}\label{eq:factor-only-hilbertian}
 \max\left\{c_X,c_Y,\frac{2}{c_Xc_Y}\right\}
 \le \CNJ(X\otpi Y)\le2.
\end{equation}
Thus, under this structural hypothesis, the auxiliary Euclidean-plane indices in the general projective lower bound can be eliminated from the final estimate.
\end{corollary}

\begin{proof}
The metric mapping property gives an isomorphism
\[
 T\otimes S:X\otpi Y\longrightarrow H\otpi K
\]
whose distortion is at most
\[
 \|T\|\,\|T^{-1}\|\,\|S\|\,\|S^{-1}\|
 \le \sqrt{c_Xc_Y}.
\]
The rescaling argument for the $\CNJ$ distortion estimate does not use finite
dimensionality, and therefore
\[
 \CNJ(X\otpi Y)
 \ge \frac{\CNJ(H\otpi K)}{c_Xc_Y}
 =\frac{2}{c_Xc_Y}
\]
by \cref{thm:hilbert-cert}. The two factor terms follow from the canonical
factor embeddings, and the universal upper bound is $2$.
\end{proof}

\begin{corollary}\label{cor:worst-plane-rigidity}
Let $X,Y$ be non-zero Banach spaces.
\begin{enumerate}[label=\textup{(\roman*)}]
\item If $X$ and $Y$ are infinite-dimensional, then
\[
 \Theta_2(\Lop(X,Y^*))=\sqrt2.
\]
\item If $H$ and $K$ are Hilbert spaces with $\dim H,\dim K\ge2$, then
\[
 \Theta_2(\Lop(H,K))=\sqrt2.
\]
\item If
\[
 \Theta_2(\Lop(X,Y^*))<\sqrt2,
\]
then $X\otpi Y$ is uniformly non-square.
\end{enumerate}
\end{corollary}

\begin{proof}
For (i), \cref{thm:rigidity,thm:bilateral} give
\[
 2=\CNJ(X\otpi Y)\le \Theta_2(\Lop(X,Y^*))^2\le2,
\]
so equality holds throughout.  For (ii), combine \cref{thm:hilbert-cert,thm:bilateral} in the same way.  Finally, (iii) follows directly from \cref{thm:bilateral}, because then
\[
 \CNJ(X\otpi Y)\le \Theta_2(\Lop(X,Y^*))^2<2,
\]
and $\CNJ<2$ is equivalent to uniform non-squareness.
\end{proof}

\section{Renorming, Hilbert thresholds and superreflexivity}

\REV{The value $2$ is not an isomorphic invariant. The relevant renorming
envelope is}
\[
 \tCNJ(Z)=\inf\{\CNJ(Z,|\!|\!|\cdot|\!|\!|):
 |\!|\!|\cdot|\!|\!|\text{ is an equivalent norm on }Z\}.
\]
\REV{Enflo's renorming theorem and the results of Kato--Takahashi imply}
\begin{equation}\label{eq:renorm-super}
 \tCNJ(Z)<2\quad\Longleftrightarrow\quad Z\text{ is superreflexive}
\end{equation}
\cite{Enflo1972,KatoTakahashi1997}.

The factor inequality has the following renorming form.

\begin{proposition}\label{prop:restriction}
Let $W=X\widehat\otimes_\alpha Y$, $\alpha\in\{\eps,\pi\}$, and let $|\!|\!|\cdot|\!|\!|$ be any norm on $W$ equivalent to the canonical tensor norm.  Fix $x_0\in S_X$, $y_0\in S_Y$, and define
\[
 |x|_X=|\!|\!|x\otimes y_0|\!|\!|,
 \qquad |y|_Y=|\!|\!|x_0\otimes y|\!|\!|.
\]
Then these are equivalent factor norms and
\[
 \max\{\CNJ(X,|\cdot|_X),\CNJ(Y,|\cdot|_Y)\}
 \le\CNJ(W,|\!|\!|\cdot|\!|\!|).
\]
\end{proposition}

\begin{proof}
Equivalence follows from the equivalence of the tensor norms and the crossnorm identity on elementary tensors.  With the restricted norms, both canonical factor maps are isometries, so subspace monotonicity applies.
\end{proof}

\begin{corollary}\label{cor:isomorphic-factor}
For non-zero Banach spaces $X,Y$ and $\alpha\in\{\eps,\pi\}$,
\[
\max\{\tCNJ(X),\tCNJ(Y)\}
 \le \tCNJ(X\widehat\otimes_\alpha Y)\le2.
\]
\end{corollary}

\begin{proof}
Apply \cref{prop:restriction} to an arbitrary equivalent norm on $X\widehat\otimes_\alpha Y$.  The induced norms on the canonical copies of $X$ and $Y$ are particular equivalent factor norms, so their von Neumann--Jordan constants dominate $\tCNJ(X)$ and $\tCNJ(Y)$, respectively.  Taking the infimum over all equivalent tensor norms gives the lower bound; the upper bound is universal.
\end{proof}

The genuinely tensorial obstruction is quantitative.  We first isolate its stable finite-dimensional form.

\begin{proposition}\label{prop:stable-linf}
Let $X$ and $Y$ be infinite-dimensional Banach spaces.  For every $n\in\mathbb N$ and every $\delta>0$:
\begin{enumerate}[label=\textup{(\roman*)}]
\item $X\oteps Y$ contains an $n$-dimensional subspace $G_\eps$ such that
\[
 d(G_\eps,\ell_\infty^n)<1+\delta;
\]
\item $(X\otpi Y)^*=\Lop(X,Y^*)$ contains an $n$-dimensional subspace $G_\pi$ such that
\[
 d(G_\pi,\ell_\infty^n)<1+\delta.
\]
\end{enumerate}
Consequently, $\ell_\infty$ is finitely representable in $X\oteps Y$ and in $(X\otpi Y)^*$.
\end{proposition}

\begin{proof}
Choose $\eta>0$ so that $(1+\eta)^2<1+\delta$.

For the injective product, Dvoretzky's theorem gives $n$-dimensional subspaces $E\subset X$ and $F\subset Y$ and isomorphisms $A:\ell_2^n\to E$ and $B:\ell_2^n\to F$ with
\[
 \|A\|\|A^{-1}\|<1+\eta,
 \qquad
 \|B\|\|B^{-1}\|<1+\eta.
\]
The diagonal map
\[
 W:\ell_\infty^n\longrightarrow \ell_2^n\oteps\ell_2^n,
 \qquad
 W(a_1,\ldots,a_n)=\sum_{j=1}^n a_j e_j\otimes e_j,
\]
is an isometry.  Since the injective norm preserves subspaces, $(A\otimes B)W$ is an embedding into $X\oteps Y$ with distortion less than $(1+\eta)^2<1+\delta$.

For the projective dual, apply Dvoretzky's theorem to $X^*$ and $Y^*$.  Choose an $n$-dimensional subspace $E\subset X^*$, an isomorphism $A:\ell_2^n\to E$, an $n$-dimensional subspace $F\subset Y^*$, and an isomorphism $B:\ell_2^n\to F$, all with distortion less than $1+\eta$.  The restriction map
\[
 Q_E:X\longrightarrow E^*,
 \qquad Q_E(x)(x^*)=x^*(x),
\]
is a metric quotient.  Using $A^*:E^*\to\ell_2^n$ as the Euclidean identification, define, for $a=(a_j)\in\ell_\infty^n$,
\[
 \Phi(a)=B D_a A^*Q_E\in\Lop(X,Y^*),
 \qquad D_a e_j=a_j e_j.
\]
After the harmless rescaling of $A$ and $B$ used in the Banach--Mazur distance, the metric quotient property and the estimates for $A,B$ give
\[
 \frac{1}{(1+\eta)^2}\|a\|_\infty
 \le \|\Phi(a)\|
 \le \|a\|_\infty.
\]
Thus the distortion is less than $(1+\eta)^2<1+\delta$.  The final assertion follows because every finite-dimensional subspace of $\ell_\infty$ embeds almost isometrically into some $\ell_\infty^N$.
\end{proof}

The following quantitative lemma converts these stable square blocks into an obstruction for every equivalent norm.

\begin{lemma}\label{lem:hadamard}
Suppose $c=\CNJ(E)<2$, let $N=2^m$, and let $p\in\{1,\infty\}$.  If
\[
 T:\ell_p^N\to E
\]
is an isomorphic embedding, then
\[
 \|T\|\|T^{-1}\|\ge\left(\frac2c\right)^{m/2}.
\]
Consequently, a space that contains either $\ell_1^n$ or $\ell_\infty^n$ uniformly for all $n$ must have von Neumann--Jordan constant two.
\end{lemma}

\begin{proof}
Let $H_m=(h_{ij})$ be the Sylvester--Hadamard matrix of order $N$.  On $E^N$ use the Hilbertian sum norm.  Iterating
\[
 \|u+v\|^2+\|u-v\|^2\le2c(\|u\|^2+\|v\|^2)
\]
through the $m$ Hadamard stages gives, for $x=(x_j)_{j=1}^N$,
\[
 \|H_mx\|_2\le(2c)^{m/2}\|x\|_2.
\]
Since $H_m^2=NI$, applying the same estimate to $H_mx$ yields
\[
 \|H_mx\|_2\ge\frac{N}{(2c)^{m/2}}\|x\|_2.
\]

Put $x_j=Te_j$.  If $p=\infty$, then
\[
 \|x\|_2\ge\frac{\sqrt N}{\|T^{-1}\|},
 \qquad
 \|H_mx\|_2^2
 =\sum_{i=1}^N\|T(h_{ij})_{j=1}^N\|^2
 \le N\|T\|^2,
\]
because each Hadamard row has $\ell_\infty$-norm one.  The lower Hadamard estimate therefore gives
\[
 \|T\|\|T^{-1}\|
 \ge\frac{N}{(2c)^{m/2}}
 =\left(\frac2c\right)^{m/2}.
\]

If $p=1$, then
\[
 \|x\|_2\le\sqrt N\,\|T\|,
\]
while every Hadamard row has $\ell_1$-norm $N$, so
\[
 \|H_mx\|_2
 \ge\frac{N\sqrt N}{\|T^{-1}\|}.
\]
Combining this with the upper Hadamard estimate gives the same lower bound for
$\|T\|\|T^{-1}\|$.  Since $c<2$, the bound tends to infinity with $m$.
\end{proof}

The preceding quantitative estimate immediately yields the following
inversion and finite-block consequences.

\begin{corollary}\label{cor:hadamard-inversion}
Let $N=2^m$ with $m\ge1$, let $p\in\{1,\infty\}$, and let
\[
 T:\ell_p^N\longrightarrow E
\]
be an isomorphic embedding with distortion $D=\|T\|\|T^{-1}\|$.  Then
\begin{equation}\label{eq:hadamard-inverted}
 \CNJ(E)\ge 2D^{-2/m}.
\end{equation}
Consequently, if for an unbounded set of integers $m$ there are embeddings
$T_m:\ell_p^{2^m}\to E$ with distortions $D_m$ satisfying
\[
 \log D_m=o(m),
\]
then $\CNJ(E)=2$.  Equivalently, along powers of two it is enough to have distortions $N^{o(1)}$; uniform containment is not required.
\end{corollary}

\begin{proof}
Put $c=\CNJ(E)$.  If $c=2$ there is nothing to prove.  Otherwise \cref{lem:hadamard} gives
\[
 D\ge\left(\frac2c\right)^{m/2},
\]
and rearranging yields \eqref{eq:hadamard-inverted}.  If $\log D_m=o(m)$, the right-hand side of \eqref{eq:hadamard-inverted} tends to two, while the universal upper bound is two.
\end{proof}

\begin{theorem}\label{thm:ordinary-euclidean-blocks}
Let $N=2^m$ with $m\ge1$.  Then
\begin{align*}
 \CNJ(X\oteps Y)
 &\ge\max\left\{\CNJ(X),\CNJ(Y),
  2\bigl(\eta_N(X)\eta_N(Y)\bigr)^{-2/m}\right\},\\
 \CNJ(X\otpi Y)
 &\ge\max\left\{\CNJ(X),\CNJ(Y),
  2\bigl(\eta_N(X^*)\eta_N(Y^*)\bigr)^{-2/m}\right\}.
\end{align*}
As above, an infinite Euclidean index is interpreted as contributing the value zero to the corresponding last term.  For $N=2$ these estimates reduce to the rank-two bounds in \cref{eq:epsbound,eq:pibound}.
\end{theorem}

\begin{proof}
Fix $D_X>\eta_N(X)$ and $D_Y>\eta_N(Y)$ and choose $N$-dimensional Euclidean sections of $X$ and $Y$ with distortions below $D_X$ and $D_Y$.  The injective metric mapping property transports the diagonal copy of $\ell_\infty^N$ in $\ell_2^N\oteps\ell_2^N$ into $X\oteps Y$ with distortion below $D_XD_Y$.  Apply \cref{cor:hadamard-inversion} and let $D_X\downarrow\eta_N(X)$, $D_Y\downarrow\eta_N(Y)$.

For the projective estimate choose instead $N$-dimensional Euclidean sections $E\subset X^*$ and $F\subset Y^*$ with distortions below $D_X>\eta_N(X^*)$ and $D_Y>\eta_N(Y^*)$.  The restriction map $Q_E:X\to E^*$ is a metric quotient.  Exactly as in the projective-dual construction of \cref{prop:stable-linf}, the diagonal operator map embeds $\ell_\infty^N$ into
\[
 (X\otpi Y)^*=\Lop(X,Y^*)
\]
with distortion below $D_XD_Y$. Apply \cref{cor:hadamard-inversion} in the
dual space and use self-duality of $\CNJ$. The factor terms follow from the
canonical factor embeddings.
\end{proof}

\begin{theorem}\label{thm:isomorphic-saturation}
If $X$ and $Y$ are infinite-dimensional and $\alpha\in\{\eps,\pi\}$, then every norm equivalent to the canonical norm on $X\widehat\otimes_\alpha Y$ has von Neumann--Jordan constant two.  Equivalently,
\[
 \tCNJ(X\widehat\otimes_\alpha Y)=2.
\]
\end{theorem}

\begin{proof}
By \cref{prop:stable-linf}, $X\oteps Y$ contains $\ell_\infty^n$ uniformly for all $n$, while the projective tensor structure provides the corresponding uniform $\ell_1^n$ blocks in $X\otpi Y$.  Passing to an equivalent norm changes all finite-dimensional distortions by at most a fixed multiplicative constant, so the corresponding uniform containment persists.  The two-sided Hadamard obstruction in \cref{lem:hadamard} therefore forces every equivalent norm on either tensor product to have von Neumann--Jordan constant two.
\end{proof}

The following result strengthens the qualitative dichotomy to an exact numerical identity for the isomorphic envelope .

\begin{theorem}\label{thm:complete-renorming}
For non-zero $X,Y$ and $\alpha\in\{\eps,\pi\}$,
\begin{equation}\label{eq:exact-isomorphic-formula}
 \tCNJ(X\widehat\otimes_\alpha Y)=
 \begin{cases}
  2, & \dim X=\dim Y=\infty,\\
  \tCNJ(Y), & \dim X<\infty,\\
  \tCNJ(X), & \dim Y<\infty.
 \end{cases}
\end{equation}
When both factors are finite-dimensional, the last two lines agree and all three relevant finite-dimensional envelopes equal one.  Consequently,
\[
 \tCNJ(X\widehat\otimes_\alpha Y)<2
\]
if and only if
\[
 \tCNJ(X)<2,\qquad \tCNJ(Y)<2,
 \qquad \min\{\dim X,\dim Y\}<\infty.
\]
Equivalently, the tensor product is superreflexive if and only if both factors are superreflexive and one factor is finite-dimensional.
\end{theorem}

\begin{proof}
The lower bound
\[
 \tCNJ(X\widehat\otimes_\alpha Y)
 \ge\max\{\tCNJ(X),\tCNJ(Y)\}
\]
is \cref{cor:isomorphic-factor}.  If both factors are infinite-dimensional, \cref{thm:isomorphic-saturation} gives the first line of \eqref{eq:exact-isomorphic-formula}.

Suppose next that $\dim X=N<\infty$.  Then $\tCNJ(X)=1$, so the preceding lower bound gives
\[
 \tCNJ(X\widehat\otimes_\alpha Y)\ge\tCNJ(Y).
\]
Fix $\delta>0$ and choose an equivalent norm $|\cdot|_Y$ on $Y$ such that
\[
 \CNJ(Y,|\cdot|_Y)\le\tCNJ(Y)+\delta.
\]
After fixing a basis $(e_i)_{i=1}^N$ of $X$, write uniquely
$w=\sum_{i=1}^N e_i\otimes y_i$ and define
\[
 |\!|\!|w|\!|\!|_2
 =\left(\sum_{i=1}^N|y_i|_Y^2\right)^{1/2}.
\]
Because $X$ is finite-dimensional, this norm is equivalent to either canonical tensor norm.  Coordinatewise application of the von Neumann--Jordan inequality in $(Y,|\cdot|_Y)$ gives
\[
 \CNJ(X\widehat\otimes_\alpha Y,|\!|\!|\cdot|\!|\!|_2)
 \le \CNJ(Y,|\cdot|_Y)
 \le\tCNJ(Y)+\delta.
\]
Taking the infimum over tensor-product renormings and then letting $\delta\downarrow0$ proves
\[
 \tCNJ(X\widehat\otimes_\alpha Y)=\tCNJ(Y).
\]
The case $\dim Y<\infty$ is symmetric.  The final characterization follows from \eqref{eq:renorm-super}.
\end{proof}

\begin{corollary}\label{cor:iterated-isomorphic-formula}
Let $X_1,\ldots,X_s$ be non-zero Banach spaces, let $\alpha_j\in\{\eps,\pi\}$ for $1\le j<s$, and define recursively
\[
 Z_1=X_1,\qquad Z_{j+1}=Z_j\widehat\otimes_{\alpha_j}X_{j+1}.
\]
Then
\[
 \tCNJ(Z_s)=
 \begin{cases}
  1, & \text{all }X_j\text{ are finite-dimensional},\\
  \tCNJ(X_k), & \text{exactly one factor }X_k\text{ is infinite-dimensional},\\
  2, & \text{at least two factors are infinite-dimensional}.
 \end{cases}
\]
\end{corollary}

\begin{proof}
Apply \cref{thm:complete-renorming} inductively.  Before the first infinite-dimensional factor appears, the accumulated tensor product is finite-dimensional and has envelope one.  Tensoring the unique infinite-dimensional factor with any number of non-zero finite-dimensional factors leaves its envelope unchanged by the second and third lines of \eqref{eq:exact-isomorphic-formula}.  Once a second infinite-dimensional factor appears, the first line of \eqref{eq:exact-isomorphic-formula} forces the value two, and all subsequent tensoring steps preserve that value.
\end{proof}

\begin{remark}\label{rem:short-renorming-proof}
The factor inequality does not by itself prove either the qualitative dichotomy or the exact formula \eqref{eq:exact-isomorphic-formula}; its role is to make the factor lower bounds automatic.  The genuinely tensorial obstruction is the persistence of uniformly large $\ell_\infty^n$ or $\ell_1^n$ blocks in \cref{thm:isomorphic-saturation}, while the finite-factor equality in \eqref{eq:exact-isomorphic-formula} uses the explicit Hilbertian sum renorming from the proof of \cref{thm:complete-renorming}.

If one instead takes the established tensor-product superreflexivity theorem as an input, then only the strict-threshold part has the shorter corollary proof
\[
 \tCNJ(X\widehat\otimes_\alpha Y)<2
 \overset{\eqref{eq:renorm-super}}{\Longleftrightarrow}
 X\widehat\otimes_\alpha Y\text{ is superreflexive}
\]
\[
 \Longleftrightarrow
 X,Y\text{ are superreflexive and }
 \min\{\dim X,\dim Y\}<\infty
 \overset{\eqref{eq:renorm-super}}{\Longleftrightarrow}
 \begin{gathered}
  \tCNJ(X)<2,\quad \tCNJ(Y)<2,\\
  \min\{\dim X,\dim Y\}<\infty.
 \end{gathered}
\]
Here the middle equivalence is the theorem of Rueda Zoca, equivalently the formulation obtained from the type--cotype results of Bu--Dowling \cite{RuedaZoca2025,BuDowling2026}.  This shortcut does not recover the exact numerical finite-factor identity or the iterated formula in \cref{cor:iterated-isomorphic-formula}; those are genuinely stronger than the qualitative superreflexivity dichotomy.
\end{remark}

We conclude the discussion by relating the tensor geometry developed above to Hilbertian and weak-Hilbert structures, which provide a natural endpoint for these investigations.

The preceding superreflexivity theorem also gives an efficient entry point to Pisier's weak Hilbert geometry.  We use the following standard geometric formulation.

\begin{definition}\label{def:weak-hilbert}
A Banach space $W$ is called \emph{weak Hilbert} if there are constants $\delta\in(0,1]$ and $C\ge1$ such that every finite-dimensional subspace $E\subset W$ contains a subspace $F\subset E$ for which
\[
 \dim F\ge\delta\dim E,
 \qquad d(F,\ell_2^{\dim F})\le C,
\]
and there is a bounded projection $P:W\to F$ with $\|P\|\le C$.
\end{definition}

This is one of Pisier's equivalent descriptions of weak Hilbert spaces.  We shall use only the standard facts that the class is invariant under isomorphism, is stable under closed subspaces and finite direct sums, and every weak Hilbert space is superreflexive; see \cite{Pisier1988,Pisier1989}.

\begin{theorem}\label{thm:metric-hilbert-rigidity}
Let $X,Y$ be non-zero Banach spaces and let $\alpha\in\{\eps,\pi\}$.  Then
\[
 X\widehat\otimes_\alpha Y\text{ is a Hilbert space with its canonical tensor norm}
\]
if and only if one factor is one-dimensional and the other factor is a Hilbert space.
\end{theorem}

\begin{proof}
If, say, $\dim Y=1$, then the canonical map $x\mapsto x\otimes y_0$, for $y_0\in S_Y$, is an onto isometry from $X$ onto either completed tensor product, so the sufficiency is immediate.

Conversely, suppose $X\widehat\otimes_\alpha Y$ is Hilbert.  The canonical factor embeddings are isometries, hence $X$ and $Y$ themselves are Hilbert spaces.  If both had dimension at least two, \cref{thm:hilbert-cert} would give
\[
 \CNJ(X\widehat\otimes_\alpha Y)=2,
\]
whereas a Hilbert norm has von Neumann--Jordan constant one.  Thus one factor must be one-dimensional.
\end{proof}

\begin{theorem}\label{thm:weak-hilbert-tensor}
	Let $X,Y$ be non-zero Banach spaces and let
	$\alpha\in\{\varepsilon,\pi\}$. Then
	\[
	X\widehat{\otimes}_{\alpha}Y
	\quad\text{is weak Hilbert}
	\]
	if and only if one factor is finite-dimensional and the other
	factor is weak Hilbert.
\end{theorem}

\begin{proof}
	Assume first that
	$Z=X\widehat{\otimes}_{\alpha}Y$ is weak Hilbert.
	Since the canonical copies of $X$ and $Y$ are closed isometric
	subspaces of $Z$, the stability of weak Hilbert spaces under closed
	subspaces implies that both $X$ and $Y$ are weak Hilbert
	\cite{Pisier1988,Pisier1989}. In particular, $X$, $Y$, and $Z$
	are superreflexive. The tensor-product superreflexivity theorem of
	Rueda Zoca \cite{RuedaZoca2025} therefore yields
	\[
	\min\{\dim X,\dim Y\}<\infty.
	\]
	Hence one factor is finite-dimensional and the other is weak Hilbert.
	
	Conversely, suppose, for instance, that $\dim X=N<\infty$ and that
	$Y$ is weak Hilbert. Fix a basis $(e_i)_{i=1}^{N}$ of $X$. The map
	\[
	\Phi:Y^N\longrightarrow X\otimes Y,
	\qquad
	\Phi(y_1,\ldots,y_N)=\sum_{i=1}^{N}e_i\otimes y_i,
	\]
	is a linear bijection. Since $X$ is finite-dimensional, either
	canonical tensor norm is equivalent, through $\Phi$, to a fixed
	product norm on $Y^N$; consequently,
	\[
	X\widehat{\otimes}_{\alpha}Y\cong Y^N.
	\]
	Finite direct sums of weak Hilbert spaces are weak Hilbert, and the
	property is invariant under isomorphisms
	\cite{Pisier1988,Pisier1989}. Thus
	$X\widehat{\otimes}_{\alpha}Y$ is weak Hilbert.
	The case $\dim Y<\infty$ is symmetric.
\end{proof}

\begin{corollary}\label{cor:weak-hilbert-maximal-cnj}
Let $W$ be an infinite-dimensional weak Hilbert space and let $F$ be a finite-dimensional Hilbert space with $\dim F\ge2$.  Then, for $\alpha\in\{\eps,\pi\}$,
\[
 F\widehat\otimes_\alpha W\text{ is weak Hilbert},
 \qquad
 \CNJ(F\widehat\otimes_\alpha W)=2.
\]
\end{corollary}

\begin{proof}
The weak-Hilbert assertion follows from \cref{thm:weak-hilbert-tensor}.  Since both $F$ and $F^*$ have OCP, the mixed-dimensional criterion \cref{thm:ocp}, together with symmetry of the two tensor norms, gives the endpoint identity for both $\eps$ and $\pi$.
\end{proof}

\begin{remark}\label{rem:weak-hilbert-gate}
\cref{thm:metric-hilbert-rigidity,thm:weak-hilbert-tensor} isolate two different thresholds.  Metric Hilbert structure survives tensorization only through a one-dimensional partner, whereas weak Hilbert structure survives precisely through a finite-dimensional partner.  Taking $F=\ell_2^2$ and $W=H$ infinite-dimensional Hilbert gives an especially transparent example: both $\ell_2^2\oteps H$ and $\ell_2^2\otpi H$ are isomorphic to the Hilbert space $H\oplus H$, yet their canonical tensor norms have $\CNJ=2$ by \cref{cor:weak-hilbert-maximal-cnj}.  Thus maximal von Neumann--Jordan defect is compatible even with a Hilbertizable tensor product.
\end{remark}

\section*{Declarations}
\noindent\textbf{Data availability.} No datasets were generated or analyzed in this work.\\
\textbf{Conflict of interest.} The author declares no conflict of interest.

\section*{Acknowledgements}
The authors thank the anonymous referees for their valuable comments.
Thanks to all the members of the Functional Analysis Research team of the College of Mathematics and Statistics of Anqing Normal University for their discussion and correction of the difficulties and errors encountered in this paper.

\end{document}